\documentclass[11pt]{article}
\usepackage{amsmath,amsthm,amsfonts,amssymb,bm,mathrsfs,wasysym}
\usepackage{epsfig}
\usepackage[usenames]{color}
\usepackage{verbatim}
\usepackage{hyperref}
\usepackage{multicol}
\usepackage{comment}
\usepackage{float}
\usepackage{graphicx}
\usepackage[utf8]{inputenc}

\parskip \medskipamount
\newtheorem{theorem}{Theorem}[section]

\numberwithin{equation}{section}
\newtheorem{lemma}[theorem]{Lemma}
\newtheorem{proposition}[theorem]{Proposition}

\numberwithin{equation}{section}

\def\N{\mathbb{N}}
\def\Z{\mathbb{Z}}

\def\C{\mathcal{C}}

\def\S{\mathcal{S}}

\def\A{\mathcal{A}}
\def\cR{\mathcal{R}}

\def\B{\mathcal{B}}

\def\bE{\mathbb{E}}

\def\bP{\mathbb{P}}

\def\bP{\mathbb{P}}

\allowdisplaybreaks

\newcommand{\1}{{\text{\Large $\mathfrak 1$}}}

\def\bs{\backslash}

\def\capa{\text{cap}}

\renewcommand{\emptyset}{\varnothing}
\renewcommand{\phi}{\varphi}
\renewcommand{\epsilon}{\varepsilon}
\def\tilde{\widetilde}

\def\reff#1{(\ref{#1})}

\begin{document}
\title{\bf Large Deviations for  Intersections of Random Walks}

\author{Amine Asselah \thanks{
LAMA, Univ Paris Est Creteil, Univ Gustave Eiffel, UPEM, CNRS, F-94010, Cr\'eteil, France; amine.asselah@u-pec.fr} \and
Bruno Schapira\thanks{Aix-Marseille Universit\'e, CNRS, Centrale Marseille, I2M, UMR 7373, 13453 Marseille, France;  bruno.schapira@univ-amu.fr} 
}
\date{}
\maketitle
\begin{abstract}
We prove a Large Deviations Principle for the number of 
intersections of two independent infinite-time ranges in 
dimension five and more, improving upon the moment bounds of Khanin, Mazel, Shlosman and Sina\"i \cite{KMSS94}. 
This settles, in the discrete setting, a conjecture of van den Berg, Bolthausen and den Hollander 
\cite{BBH04}, who analyzed this question for the Wiener sausage in finite-time horizon. The proof builds on their result (which was resumed in the discrete setting by Phetpradap \cite{Phet12}), and combines it with a series of tools that were developed in recent works of the authors \cite{AS17,AS19a,AS20}. Moreover, we show that most of the intersection 
occurs in a single box where both walks realize an occupation density of order one.  \\

\noindent \underline{Keywords and phrases:} Random Walk; Range; Intersections; Large Deviations Principle. \\
MSC 2010 \underline{subject classifications:} Primary 60F10, 60G50.
\end{abstract}

\section{Introduction}\label{sec-intro}
\subsection{Overview and results}
In 1921, P\'olya \cite{P1921} presents his recurrence theorem,
inspired by some counter-intuitive observation on the large number of intersections
two random walkers in a park would make. A hundred
years later, the study of intersections of random walks is still active, and 
produces perplexing problems. This paper is devoted to estimating deviations for the
number of sites two infinite trajectories both visit, 
when dimension is five or larger.

It is known since the work of Erd\"os and Taylor \cite{ET60}, that the number of intersections of two independent random walk ranges on $\Z^d$ is almost surely 
infinite if $d\le 4$, and finite if $d\ge 5$. In 1994,   
Khanin, Mazel, Shlosman and Sina\"i \cite{KMSS94} obtain the following bounds in dimension $d\ge 5$: for any $\varepsilon>0$, and all $t$ large enough, 
\begin{equation}\label{KMSS}
\exp(- t^{1-\frac 2d +\epsilon})  \le \bP(|\cR_\infty \cap \tilde \cR_\infty|>t ) \le \exp(t^{1-\frac 2d-\epsilon}),
\end{equation}
where $\cR_\infty$ and $\tilde \cR_\infty$ denote two independent ranges.  
About ten years later, van den Berg, Bolthausen and den Hollander \cite{BBH04} prove a Large Deviations Principle for the Wiener sausage (the continuous counterpart of the range), 
in {\it a finite-time horizon}. Their result was resumed in the discrete setting by Phetpradap \cite{Phet12} and reads as follows: for any $b>0$, there exists a positive constant $\mathcal I(b)$, such that 
\begin{equation}\label{BBH.Phet}
\lim_{t\to \infty}   \frac 1{t^{1-\frac 2d}} \log  \bP(|\cR_{bt}\cap \tilde \cR_{bt}|>t) = -\mathcal I(b), 
\end{equation}
where $\cR_{bt}$ and $\tilde \cR_{bt}$ denote the ranges of two independent walks up to time $\lfloor bt\rfloor $. Furthermore, 
through an analysis of the variational formula of the rate function, 
the authors of \cite{BBH04} show that $\mathcal I(b)$ 
reaches a plateau and conjecture that 
the rate function for the infinite-time problem coincides
with the value of $\mathcal I$ at the plateau.  
Our first result confirms this conjecture. The ranges of two independent simple random walks is denoted $\{\cR_n, n\in \N\cup\{\infty\}\}$ and $\{\widetilde \cR_n,  
n\in \N\cup \{\infty\}\}$.

\begin{theorem} \label{theo.main}
Assume $d\ge 5$. 
The following limit exists and is positive:  
\begin{equation}\label{main.limit}
\mathcal I_\infty:=\lim_{t\to \infty} -\frac{1}{t^{1-\frac 2d}} \log \bP(|\cR_\infty \cap \tilde \cR_\infty|>t). 
\end{equation}
Moreover, there exists $b_*>0$, such that for all $b>b_*$, 
\begin{equation}\label{limit.finitetime}
\mathcal I_\infty = \mathcal I(b)  = \lim_{t\to \infty} -\frac{1}{t^{1-\frac 2d}} \log \bP(|\cR_{bt} \cap \tilde \cR_{bt}|>t).
\end{equation}  
\end{theorem}
For $\mathcal I_\infty$ and $b_*$, \cite{BBH04} presents variational
formulas whose thorough study leads
to a rich and precise phenomenology. Namely, that the two walks adopt the same strategy, the so-called Swiss cheese during a time $b_*t$, in a ball-like region whose 
volume should be of order $t$, leaving holes everywhere of size order $1$. After time $b_*t$, the two walks would roam as typical random walks.

Our second result shows that a fraction arbitrarily close to one of the desired number of intersections occurs in a box with volume of order $t$. 
To state the result, define  $Q(x,r) := [x-r/2,x+r/2)^d$, 
for $x\in \Z^d$, and $r>0$. 

\begin{theorem}\label{theo.scenario}
For any $\epsilon>0$, there exists a constant $L=L(\varepsilon)>0$, such that 
\begin{equation}\label{scenario}
\lim_{t\to \infty} \bP(\exists x\in \Z^d\, :\,  |\cR_\infty \cap \widetilde \cR_\infty\cap Q(x,Lt^{1/d})| >(1-\epsilon)t\mid |\cR_\infty \cap \tilde \cR_\infty|>t )= 1. 
\end{equation}
\end{theorem}
Our proof provides some bound on $L(\epsilon)$, which is (stretched) exponential in $1/\epsilon$. We note that it is expected that $L$ should indeed depend on 
$\epsilon$, since the Swiss cheese is delocalized, see \cite{BBH04}. Concerning the (random) site $X(t,\epsilon)$ realizing the centering of the box appearing in the statement of Theorem \ref{theo.scenario}, 
not much is known. Our proof yields tightness of $X(t,\epsilon)/t^{1/d}$.

Sznitman in \cite{S17} formalized precisely the picture of Swiss cheese using a tilted version of the Random Interlacements, but so far no rigorous link has been established with the large deviations for the volume of the range nor for the intersection of two ranges.

Our techniques are robust enough to consider other natural functionals of two ranges, which do not seem to be tractable by moment methods, as in \cite{KMSS94}. 
In particular in \cite{AS19b} we consider the functional $\chi_\C(\cdot,\cdot)$ defined for finite subsets $A,B\subseteq \Z^d$, by   
$$\chi_\C(A,B) =  \capa(A) + \capa(B)- \capa(A\cup B),$$
where $\capa(A):=\sum_{x\in A} \bP_x(\cR[1,\infty)\cap A= \emptyset)$, denotes the capacity of $A$.  
It turns out that this definition may be extended to infinite subsets. Indeed, one has for any finite $A,B\subseteq \Z^d$,  
$$\chi_\C(A,B) \le \chi(A,B):= 2\sum_{x\in A} \sum_{y\in B} \bP_x(\cR[1,\infty)\cap A = \emptyset)\cdot G(y-x)\cdot \bP_y(\cR[1,\infty)\cap B = \emptyset), $$
and it makes sense to consider $\chi(\cR_\infty,\tilde \cR_\infty)$. In \cite{AS19b}, we show using similar arguments as here that in dimension $d\ge 7$, for some positive constants $c_1,c_2$, and all $t$ large enough,  
$$\exp(-c_1 t^{1-\frac 2{d-2}}) \le \bP(\chi(\cR_\infty,\tilde \cR_\infty) > t) \le \exp(-c_2 t^{1-\frac 2{d-2}}).$$
These bounds are used in turn to derive a moderate deviations principle for the capacity of the range in the Gaussian regime.

Interestingly, a related object, the mutual intersection local time defined by 
$$J_\infty := \sum_{i=0}^\infty \sum_{j=0}^\infty \1\{S_i = \tilde S_j\},$$
has a stretched exponential tail with a different exponent. Indeed, Khanin et al. in \cite{KMSS94}, also show that for some positive constants $c$ and $c'$, for all $t$ large enough,  
$$\exp(-c \sqrt t ) \le \bP(J_\infty >t) \le \exp(-c'\sqrt t).$$
Chen and M\"orters \cite{CM09} then prove that the limit of $ t^{-1/2}\cdot \log \bP(J_\infty >t)$  
exists and has a nice variational representation. Our proofs allow to consider some intermediate quantity, the time spent by one walk on the range of the other walk, and show that its tail distribution 
has the same speed of decay as the intersection of  two ranges. 
More precisely, consider two independent walks $S$ and $\tilde S$, and denote by $\tilde \ell_\infty$ the local times associated to  $\tilde S$ (see below for a definition). 
\begin{proposition}\label{prop.localtime}
There exists two positive constants $c_1$ and $c_2$, such that for any $t>0$, 
\begin{equation}\label{bornes.localtime}
\exp(-c_1t^{1-\frac 2d}) \le \bP(\tilde \ell_\infty (\cR_\infty) > t) \le \exp(-c_2t^{1-\frac 2d}). 
\end{equation}
Furthermore, for any $\epsilon>0$, there exists an integer $N=N(\epsilon)$, such that 
\begin{equation}\label{finite.balls.localtime}
\lim_{t\to \infty} \bP\left(\exists x_1,\dots,x_N\in \Z^d\, :\, \tilde \ell_\infty\left(\cR_\infty \cap \Big(\bigcup_{i=1}^N Q(x_i,t^{1/d})\Big)\right)> (1-\epsilon)t\mid \tilde \ell_\infty (\cR_\infty) > t\right)= 1. 
\end{equation} 
\end{proposition}
Analogous results such as Theorems \ref{theo.main} and \ref{theo.scenario} would hold for $\widetilde \ell_\infty(\cR_\infty)$, conditionally on obtaining first an analogue of \eqref{BBH.Phet} for $\tilde \ell_{bt}(\cR_{bt})$, which is presumably true, but not available at the moment.

Let us remark also that the problem we address here
has a flavor of a much studied problem of {\it random walk in random landscape},
where the random landscape is produced here by another independent walk. Here also, it
appears interesting to study a {\it quenched regime}, where one walk is frozen in a typical realization, 
whereas the second tries to hit $t$ sites of the first 
range. This problem is still untouched, and we believe that our techniques
will shed some light on it.

\subsection{Proof strategy} 
While the proof in \cite{KMSS94} used a moment method and some ingenious computations, our proof is based on more geometric arguments.

There are two parts. 
In the first one, we show that conditionally on the intersection event, with probability going to one, the whole intersection takes place in a finite number of boxes 
(as in Proposition \ref{prop.localtime} above). In the second part we use the full power of the LDP \eqref{BBH.Phet} and the concavity of the speed $t\mapsto t^{1-\frac 2d}$ to reduce the number of boxes 
where the action occurs to a unique box, which gives Theorem \ref{theo.scenario}, and then we also deduce Theorem \ref{theo.main}.

The first part is itself obtained in three steps. First we reduce the time window to a finite time interval, using that it is unlikely for one walk to intersect the range of the other 
walk after a time 
of order $\exp(\beta\cdot t^{1-\frac 2d})$, for some large $\beta$. This leaves however a lot of room for the places where the action could take place (since we recall it holds in a box with volume of order $t$ only). 
In particular decomposing space into boxes and using a union bound type argument would not work, at least not directly. Our main idea to overcome this difficulty is to 
divide space according to the occupation density of the range, which we do at different space-scales depending on the density we are considering,   
in a similar fashion as in \cite{AS19a,AS19b}. Then we use a fundamental tool from \cite{AS19a} which gives a priori bounds on the size of these regions, with the conclusion that it is only 
in those with high density (of order one) that the intersection occurs. Finally we use another recent  
result from \cite{AS20}, which bounds the probability to cover a positive fraction of any {\it fixed} union of distant boxes. When we further impose that these boxes are visited by another independent walk, one can sum over all possible centers of the boxes, and this yields some bound on the number of boxes, with volume of the right order, that are needed to cover the region where the intersection occurs.

For the second part of the proof, we decompose 
the journeys between a finite number of boxes into excursions either within one box, or joining two boxes.  
Then some surgery is applied. We cut the excursions between different boxes and replace them by excursions drawn independently with  starting points sampled according to the harmonic measure. 
This allows to compare the probability of the event when the walk realizes the intersection in $N$ different boxes, to the product of the probabilities of realizing (smaller) intersection in each of these boxes, and one can then use \eqref{BBH.Phet} to bound these probabilities. This is also where the concavity is used, to show that one box is better than many, and the surgery arguments are then used again to restaure the journeys and yield Theorem \ref{theo.main}.

\subsection{Organization}
The paper is organized as follows. In the next section we recall the main notation, and the tools that will be used in the proofs, which for the most part appeared in our previous works 
\cite{AS17,AS19a, AS19b, AS20}. In Section \ref{sec-tech} we give a detailed plan of the proofs of ours main results. The latter are then 
proved in the remaining sections \ref{sec.mainprop}--\ref{sec.propfiniteballs}--\ref{sec.proofstheo}.

\section{Notation and main tools}
\subsection{Notation and basic results}
Let $\{S_n\}_{n\ge 0}$ be a simple random walk on $\Z^d$.  
We denote by $\bP_x$ its law starting from $x$, which we abbreviate as $\bP$ when $x=0$. 
We mainly assume here that $d\ge 5$, yet some results hold for all $d\ge 3$, in which case we shall mention it explicitly.
For $n \in \N\cup \{\infty\}$, we write the range of the walk up to time $n$ as $\cR_n := \{S_0,\dots,S_n\}$.  
More generally for $n\le m$ two (possibly infinite) integers, we consider the range between times $n$ and $m$, defined as $\cR[n,m] := \{S_n,\dots,S_m\}$.  
For $\Lambda\subseteq \Z^d$, and $n\in \N\cup \{\infty\}$, we define the time spent in $\Lambda$ as 
$$\ell_n(\Lambda) := \sum_{k=0}^n \1\{S_k\in \Lambda\},$$
and simply let $\ell_n(z)$ be the time spent on a site $z\in \Z^d$.  
The {\bf Green's function} is defined by 
$$G(x,z) := \sum_{k=0}^\infty \bP_x(S_k=z) = \bE_x[\ell_\infty(z)].$$ 
By translation invariance one has for any $x,z\in \Z^d$, $G(x,z) = G(0,z-x)=: G(z-x)$. Thus for any $\Lambda\subset \Z^d$, and any $x\in \Z^d$, 
$$\bE_x[\ell_\infty(\Lambda)] = \bE[\ell_\infty(\Lambda-x)] =\sum_{z\in \Lambda-x}G(z-x) =: G(\Lambda-x).$$
Furthermore, there exists a constant $C>0$, such that for all $z\in \Z^d$ (see \cite{LL}), 
\begin{equation}\label{Green}
G(z) \le  \frac{C}{1+\|z\|^{d-2}}, 
\end{equation}
where $\|\cdot \|$ denotes the Euclidean norm, and for all $r>0$ and $z\in \Z^d$, 
\begin{equation}\label{bound.hit}
\bP(\cR_\infty \cap Q(z,r)\neq \emptyset) \le Cr^{d-2}G(z). 
\end{equation}

\subsection{Preliminaries}
We recall and discuss here a series of known results on which relies our proof. Most of them come from our recent works \cite{AS17,AS19a, AS19b, AS20}.

In fact the first one is older and  
shows that the 
tail distribution of the time spent in a region is controlled simply by its mean value, when starting from the worst point. We recall its short proof for completeness. 
\begin{lemma}[\cite{AC07}] \label{lem.tail}
Let $\Lambda\subseteq \Z^d$ be a (non necessarily finite) subset of $\Z^d$, $d\ge 3$. Then for any $t>0$, 
$$\bP(\ell_\infty(\Lambda)>t) \le 2 \exp\left( -\frac{t\cdot \log 2 }{2\sup_{x\in \Lambda} G(\Lambda-x)} \right).$$
\end{lemma}
\begin{proof} The result simply follows from the fact that by Markov's inequality (and the Markov property), 
the random variable $\frac{\ell_\infty(\Lambda)}{ 2\sup_{x\in \Lambda} \mathbb E_x[\ell_\infty(\Lambda)]}=\frac{\ell_\infty(\Lambda)}{ 2\sup_{x\in \Lambda} G(\Lambda-x)} $, is stochastically bounded by a geometric random variable with parameter $1/2$. 
\end{proof} 
We need also to estimate the expected time spent (or equivalently the sum of the Green's function) on the range of an independent random walk.  For this we use several facts. The first one is the following well-known simple lemma.
\begin{lemma}\label{lem.sumGreen}
There exists $C>0$, such that for any finite subset $\Lambda\subseteq \Z^d$, $d\ge 3$, one has 
$$G(\Lambda)= \sum_{z\in \Lambda}G(z) \le C |\Lambda|^{2/d}.$$
\end{lemma}
\begin{proof}
The result follows from the bound \eqref{Green}, and observing that the resulting sum is maximized (at least up to a constant) when points of $\Lambda$ are all contained in a ball of side-length of order $|\Lambda|^{1/d}$. 
\end{proof}
Now we decompose the points of the range in several subsets according to the occupation density in some neighborhoods of these points, and use that Green's function is additive in the sense that for any disjoints subsets 
$\Lambda,\Lambda'\subseteq \Z^d$, it holds $G(\Lambda \cup \Lambda') = G(\Lambda) + G(\Lambda')$. Thus we need to estimate the 
Green's function of regions with some prescribed density, which is the content of Lemma \ref{lem.tps.passe} below.    
Recall that for $r\ge 1$, and $x\in \Z^d$, we set  
$$Q(x,r) := [x-r/2, x+r/2)^d,$$ 
the cube centered at $x$ of side length $r$. 
The next result is Lemma 4.3 from \cite{AS19b}. It can be proved using a very similar argument as for the proof of Lemma \ref{lem.sumGreen}. 
\begin{lemma}[\cite{AS19b}] \label{lem.tps.passe}
Assume $d\ge 3$. There exists a constant $C>0$, such that the following holds. For any integer $r\ge 1$, any $\rho>0$, and any finite subset $\Lambda\subseteq \Z^d$, satisfying 
$$|\Lambda \cap Q(z,r) | \le \rho \cdot r^d, \quad \text{for all }z\in r\Z^d,$$
one has 
$$G(\Lambda\cap Q(0,r)^c) \le C\rho^{1-\frac 2d} |\Lambda|^{2/d}. $$ 
\end{lemma}
We now turn to estimating the number of points in the range of a random walk, around which the walk realizes a certain occupation density. For $n\in \N$, $r\ge 1$, and $\rho >0$, we define 
\begin{equation}\label{def.Rnrho}
\cR_n(r,\rho)=\{x\in \cR_n \, :\, |\cR_n\cap Q(x,r)|>\rho\cdot r^d\}. 
\end{equation}
  
\begin{theorem}[\cite{AS19a}] \label{theo-folding}
Assume $d\ge 3$.
There are positive constants $ \kappa$, and $C_0$, such that for any $n$, $r$ and $L$ positive integers and $\rho>0$,  satisfying
\begin{equation}\label{cond-folding}
\rho r^{d-2}\ge C_0\cdot \log n, 
\end{equation}
one has
\begin{equation*}
 \bP\big( |\cR_n(r,\rho)|> L\big)\le \exp\big(- \kappa\cdot
\rho^{2/d} \cdot L^{1-2/d}\big).
\end{equation*}
\end{theorem}
A weaker version of this result first appeared in \cite{AS17}, with the stronger condition $\rho r^{d-2} \ge C_0(\frac{L}{\rho r^d})^{2/d} \log n$, and the elimination of the $L$ dependence is fundamental here.

Finally the following result is used to reduce the number of boxes where most of the intersection occurs. For $r\ge 1$, some integer, we denote by $\mathcal X_r$ the collection of finite 
subsets of $\Z^d$, whose points are at distance at least $r$ from each other. For $\mathcal C\subseteq \Z^d$, we let $Q_r(\C):= \cup_{x\in \C} Q(x,r)$. 
\begin{theorem}[\cite{AS20}] \label{theo.finitecovering}
Assume $d\ge 3$. 
There exist positive constants $\kappa$ and $C$, such that for any $\rho>0$, $r\ge 1$ and $\mathcal C\in \mathcal X_{4r}$, satisfying 
$$\rho r^{d-2} > C\log |\C|,$$
one has 
$$\bP\left(\ell_\infty(Q(x,r)) > \rho r^d, \ \forall x\in \C \right) \le C\exp(-\kappa \rho\cdot \capa (Q_r(\C))). $$
\end{theorem} 
Using the well-known bound $\capa(\Lambda) > c|\Lambda|^{1-2/d}$, for any finite $\Lambda\subseteq \Z^d$, and some universal constant $c>0$, we have $\capa(Q_r(\C))\ge cr^{d-2} |\C|^{1-2/d}$. This latter bound is used later.

\section{Plan of the proof} \label{sec-tech}
Recall that we consider two independent walks $\{S_n\}_{n\ge 0}$ and $\{\widetilde S_n\}_{n\ge 0}$. All quantities associated to the second walk will be decorated with a tilde. With a slight abuse of notation we still denote by $\bP$ the law of the two walks.

The first step is to reduce the problem to a finite time horizon. 
For this we simply use a first moment bound, and the well-known fact that for any $n\ge 1$ (see \cite[Proposition 3.2.3]{Law96}), for some constant $C>0$, 
$$\bE[\tilde \ell_\infty (\cR[n,\infty))]= \sum_{z\in \Z^d}G(z) \cdot \bP(z\in \cR[n,\infty))\le C n^{\frac{4-d}{2}}.$$
Using next Markov's inequality we deduce (see also \cite[Lemma 9]{ET60} for a similar statement), 
\begin{equation}\label{finite.time}
\bP(\tilde \cR_\infty \cap \cR[n,\infty) \neq \emptyset   ) \le \bP(\tilde \ell_\infty (\cR[n,\infty))\ge 1  ) \le \bE[\tilde \ell_\infty(\cR[n,\infty))] \le C n^{\frac{4-d}{2}}. 
\end{equation} 
Thanks to this inequality, it suffices in fact to consider only the intersection of the two walks up to a time $n$ of order $\exp(\beta t^{1-2/d})$, with $\beta$ some appropriate constant.

The second step is the following proposition. Recall the definition \eqref{def.Rnrho}. 
\begin{proposition}\label{prop.main}
For any $\beta\ge 1$, there exist positive constants $c$ and $C$, such that for any $t>0$, one has with $n:= \exp(\beta t^{1-\frac 2d})$, 
\begin{equation}\label{borne.simple}
\bP\left(  \sup_{x\in \Z^d} G(\cR_n-x) > Ct^{2/d}\right) \le C\exp(-ct^{1-\frac 2d}). 
\end{equation}
Furthermore, for any $\epsilon>0$ and $K>0$, there exists $\rho=\rho(\epsilon,K,\beta)$, and $A=A(\epsilon,K,\beta)$, such that  
\begin{equation}\label{borne.simple.2}
\bP\left(  \sup_{x\in \Z^d} G\Big(\cR_n\setminus \cR_n( At^{1/d},\rho)-x\Big) > \epsilon t^{2/d}\right) \le C\exp(-Kt^{1-\frac 2d}). 
\end{equation} 
One can moreover choose $\rho$ and $A$, such that $\frac{\log \epsilon}{\log \rho}$ and $\frac{\log A}{\log(1/\epsilon)}$ remain bounded as $\epsilon \to 0$. 
\end{proposition}
The third step is to deduce that most of the intersection occurs in a finite number of boxes. 
That is we prove Proposition \ref{prop.localtime}, as well as its analogue for the mutual intersection, which we state as a separate proposition: 

\begin{proposition}\label{prop.finiteballs}
There exists two positive constants $c_1$ and $c_2$, such that for any $t>0$, 
\begin{equation}\label{bornes.intersections}
\exp(-c_1t^{1-\frac 2d}) \le \bP(|\cR_\infty\cap \tilde \cR_\infty| > t) \le \exp(-c_2t^{1-\frac 2d}). 
\end{equation}
Furthermore, for any $\epsilon>0$, there exists an integer $N=N(\epsilon)$, such that 
\begin{equation}\label{scenario.finiteballs}
\lim_{t\to \infty} \bP\left(\exists x_1,\dots,x_N\in \Z^d\, :\, |\cR_\infty \cap \tilde \cR_\infty\cap \Big(\bigcup_{i=1}^N Q(x_i,t^{1/d})\Big)|> (1-\epsilon)t\mid |\tilde \cR_\infty \cap \cR_\infty| > t\right)= 1. 
\end{equation} 
One can moreover choose $N(\epsilon)$, such that $\frac{\log N(\epsilon)}{\log (1/\epsilon)}$ remains bounded as $\epsilon\to 0$.  
\end{proposition}
The lower bound in \eqref{bornes.intersections} follows of course from \eqref{BBH.Phet}, but for the sake of completeness, we provide another independent argument based on \cite{AS17}, 
which makes the proof of Propositions \ref{prop.localtime} and \ref{prop.finiteballs} independent of \cite{BBH04}. The upper bound in \eqref{bornes.intersections} on the other hand simply follows from  Lemma \ref{lem.tail}, together with \eqref{finite.time} and \eqref{borne.simple}. Now concerning \eqref{scenario.finiteballs}, note that it 
would follow as well from \eqref{finite.time}, \eqref{borne.simple.2}  
and Lemma \ref{lem.tail}, if we could combine it with Theorem \ref{theo-folding}, since we just need to show that the set $\cR_n(A t^{1/d},\rho)$ can be covered by a finite number of cubes. 
This would be fine indeed, if we could choose the constants $A$ and $\rho$ given by \eqref{borne.simple.2} as large 
as wanted, so to satisfy the condition \eqref{cond-folding}. However, since in fact they may be small, we use instead Theorem \ref{theo.finitecovering}.

The rest of the proof relies on the results of \cite{BBH04, Phet12}, and \eqref{BBH.Phet}. We first reduce the region where most of the intersection occurs, from an arbitrary finite number of boxes to a unique, possibly enlarged, one; 
in other words we prove Theorem \ref{theo.scenario}. This part is based on the concavity of the map $t\mapsto t^{1-2/d}$, which implies that distributing the total intersection $t$ 
on more than one box increases the cost of the deviations. Note that it is crucial here to know the exact constant in the exponential, which is why we need \eqref{BBH.Phet}. 
We also use some surgery on the trajectories of the two walks; that is first a decomposition into excursions between the various boxes, 
and then a cutting/gluing argument to ensure that intersections inside each box occur in time-windows of order $t$, so to make \eqref{BBH.Phet} applicable. 
Finally, the same operation of surgery allows also to deduce Theorem \ref{theo.main} from Theorem \ref{theo.scenario} and \eqref{BBH.Phet}.

Now the end of the proof is organized as follows. We first prove Proposition \ref{prop.main} in Section \ref{sec.mainprop}. We then prove Propositions \ref{prop.localtime} and \ref{prop.finiteballs} in Section \ref{sec.propfiniteballs}, and finally we conclude the proofs of Theorems \ref{theo.main} and \ref{theo.scenario} in Section \ref{sec.proofstheo}. 


\section{Proof of Proposition \ref{prop.main}}\label{sec.mainprop}
We first introduce a decomposition of the range into subsets according to the occupation density of their neighborhoods, at different scales and bound the cardinality of each subset using Theorem \ref{theo-folding}. Then we prove \eqref{borne.simple} and \eqref{borne.simple.2} separately in Subsections \ref{subsec.borne1} and \ref{subsec.borne2} respectively. 

\subsection{Multi-scale decomposition of the range}
Our approach relies on a simple multi-scale analysis of the occupation
densities, on which space and density are scaled together. 
More precisely we introduce a sequence of densities $\{\rho_i\}_{i\ge 0}$ and associated space-scales $\{r_i\}_{i\ge 0}$ defined respectively, for any integer $i\ge 0$, by 
\begin{equation}\label{rhoi.ri}
\rho_i:=2^{-i}, \quad \text{and}\quad  \rho _i\cdot  r_i^{d-2}= C_0\log n, 
\end{equation} 
with $C_0$ the constant appearing in \eqref{cond-folding}.

It might be that on small scales, say $r_j$ for $j<i$,
the density around some point of the range {\it remains small},
whereas it overcomes $\rho_i$ at scale $r_i$. To encapsulate
this idea we define for $i\ge 1$ (recall \eqref{def.Rnrho} and note that by definition $\cR_n(r_0,\rho_0)$ is empty), 
\begin{equation}\label{def.Li}
\Lambda_i:=\cR_n(r_i,\rho_i)\bs \left(\bigcup_{1\le j<i}  \cR_n(r_j,\rho_j)\right), \quad \text{and}\quad \Lambda_i^*=\cR_n 
\bs \left(\bigcup_{1\le j<i}  \cR_n(r_j,\rho_j)\right).
\end{equation}
When dealing with these sets  
we will use two facts: on one hand for each $i\ge 1$, $\Lambda_i$ is a subset of $\cR_n(r_i,\rho_i)$, and thus Theorem \ref{theo-folding} will provide some control on its volume. 
On the other hand, using that $\Lambda_i^*\subseteq \cR_n(r_{i-1},\rho_{i-1})^c$, and by cutting a box into $2^d$ disjoint sub-boxes of side-length twice smaller, we can see that  
\begin{equation}\label{density.rj}
|\Lambda_i^*\cap Q(z,r_{i-1})| \le 2^d\rho_{i-1} r_{i-1}^d, \quad \text{for all }z\in \Z^d, \text{ and all } i>0. 
\end{equation}
Note also that since $\Lambda_i\subseteq \Lambda_i^*$, the same bounds hold for $\Lambda_i$.

By Theorem~\ref{theo-folding}, we have for some constant $\kappa>0$, for any $\lambda>0$, and any $i\ge 1$, 
\begin{equation}\label{proof-3}
\bP\big( |\Lambda_i|> \lambda\big)\le \exp(-\kappa \rho_i^{2/d}\cdot \lambda^{1-2/d}).
\end{equation}
Note also that since $|\cR_n|\le n+1$, the set $\Lambda_i$ is empty when $\rho_i r_i^d>n+1$, or equivalently when $C_0r_i^2 \log n>n+1$. In particular, for $n$ large enough, 
\begin{equation}\label{emptyset}
\Lambda_i=\emptyset, \quad \text{for all }i> (d-2)\log_2(n).
\end{equation} 
Now for $L>0$, define the good event: 
\begin{equation*}
\mathcal E_L:= \left\{|\Lambda_i|\le \rho_i^{-\frac{2}{d-2}} \cdot L  t,\quad \text{for all }i\ge 1 \right\}.
\end{equation*}
Then \eqref{proof-3} and \eqref{emptyset} show that for some constant $C>0$, 
\begin{equation}\label{prob.E}
\bP(\mathcal E_L^c) \le C\log_2(n) \exp(-\kappa (L t)^{1-\frac 2d}) \le C\exp(-\frac{\kappa}{2}\cdot (L t)^{1-\frac 2d}). 
\end{equation}

Now to motivate the definition of the sets $\Lambda_i$, and as a warmup for future computations in the next subsections, 
let us bound $\sup_{x\in \Lambda_i} G(\Lambda_i-x)$. We first write for $x\in \Lambda_i$, 
$$G(\Lambda_i-x) = G((\Lambda_i-x) \cap Q(0,r_{i-1}))  + G((\Lambda_i -x) \cap Q(0,r_{i-1})^c).$$ 
We then use Lemma \ref{lem.tps.passe} to bound the second term. This yields  
$$G((\Lambda_i -x) \cap Q(0,r_{i-1})^c)\le C \rho_{i-1}^{1-\frac 2d}\cdot |\Lambda_i|^{2/d}.$$
For the first term we use that by definition $|\Lambda_i\cap Q(x,r_j)|\le \rho_j r_j^d$, for all $j<i$. This yields 
\begin{align*}
G((\Lambda_i-x)\cap Q(0,r_{i-1}) ) & = G((\Lambda_i-x) \cap Q(0,r_0)) + \sum_{j=1}^{i-1}G((\Lambda_i-x)\cap Q(0,r_j)\bs Q(0,r_{j-1}) )\\
& \stackrel{\eqref{Green}}{\le} G(Q(0,r_0)) +C \sum_{j=1}^{i-1} \frac {\rho_j r_j^d}{r_{j-1}^{d-2}}  \stackrel{\text{Lemma } \ref{lem.sumGreen}}{\le} C\left\{r_0^2 +   \sum_{j=1}^{i-1} \frac { \log n}{r_j^{d-4}}\right\} \\
& \le C\left\{r_0^2 +  \frac {\log n}{r_0^{d-4}} \right\}\le C(\log n)^{\frac{2}{d-2}} 
\le Ct^{2/d},
\end{align*}
with a constant $C$ that is independent of $x\in \Lambda_i$. Altogether this gives 
$$\sup_{x\in \Lambda_i} G(\Lambda_i - x) \le C(t^{2/d} +\rho_{i-1}^{1-\frac 2d}|\Lambda_i|^{2/d}). $$ 
Now on the events $\mathcal E_L$,we get a bound $Ct^{2/d}$, for another constant $C$ that only depends on $L$. These bounds are however not sufficient to prove 
the results we need, since there are order $\log n$ sets $\Lambda_i$ to consider, but the idea of the proofs in the next two subsections will be similar.   

\subsection{Proof of \eqref{borne.simple}}\label{subsec.borne1}  
We claim that for some constant $C>0$, it holds 
\begin{equation}\label{goal.Grange}
\mathcal E_1\subseteq \{\sup_{x\in \Z^d} G(\cR_n-x) \le Ct^{2/d}\}.
\end{equation}
By \eqref{prob.E}, this would imply the desired result, so let us prove \eqref{goal.Grange} now.

Assume that the event $\mathcal E_1$ holds, and let us bound $\sup_{x\in \Z^d} G(\cR_n-x)$.

We fix some $x$, and divide space into concentric shells as follows: for integers $k\ge 1$, set 
$$\mathcal S_k:= Q(x,r_k)\bs Q(x,r_{k-1}),$$
and $\mathcal S_0 = Q(x,r_0)$.   
Then we use additivity to write 
$$G(\cR_n-x) = \sum_{k\ge 0} G(\mathcal S_k \cap \cR_n).$$
By Lemma \ref{lem.sumGreen} and \eqref{Green}, one has on $\mathcal E_1$, 
\begin{equation}\label{borne.S_0}
G(\cR_n\cap \S_0) \le G(\S_0)\le Cr_0^2 \le C(\log n)^{\frac{2}{d-2}} \le Ct^{2/d},
\end{equation} 
with $C$ some positive constant, whose value might change from line to line. 
Furthermore, for any $k\ge 1$, recalling \eqref{def.Li}, 
\begin{equation*}
G(\S_k \cap \cR_n) = \sum_{j=1}^{k} G\big(\S_k \cap  \Lambda_j\big) + G(\S_k \cap \Lambda_{k+1}^*). 
\end{equation*}
By \eqref{Green} and \eqref{density.rj}, one has for any $k\ge 1$, 
$$G(\S_k \cap \Lambda_{k+1}^*)\le C\cdot \frac{|\S_k \cap \Lambda_{k+1}^*|}{r_{k-1}^{d-2}} \le C\cdot \frac{\rho_kr_k^d}{r_{k-1}^{d-2}} \le C\cdot \frac{\log n}{r_k^{d-4}},$$
using also \eqref{rhoi.ri} for the last inequality. Summing over $k$ gives 
$$\sum_{k\ge 1} G(\S_k \cap \Lambda_{k+1}^*) \le C\frac{\log n}{r_0^{d-4}} \le C (\log n)^{1-\frac{d-4}{d-2}} \le C(\log n)^{\frac 2{d-2}} \le Ct^{2/d}. $$
On the other hand by Lemma \ref{lem.tps.passe}, for any $j\ge 1$, on $\mathcal E_1$, 
$$\sum_{k\ge j} G(\S_k\cap \Lambda_j) = G(\Lambda_j\cap Q(x,r_{j-1})^c) 
\le C\rho_{j-1}^{1-\frac 2d} |\Lambda_j|^{2/d} \le C\rho_j^{1-\frac 2d(1+\frac 2{d-2}) } t^{2/d}\le C\rho_j^{\frac{d-4}{d-2}}t^{2/d}. $$ 
Summing over $j\ge 1$, gives 
$$\sum_{j\ge 1} \sum_{k\ge j} G(\S_k\cap \Lambda_j) \le Ct^{2/d},$$
which concludes the proof of \eqref{goal.Grange}, and \eqref{borne.simple}.

\subsection{Proof of \eqref{borne.simple.2}} \label{subsec.borne2}
Let us give some $\epsilon$ and $K$, and then fix $L$ such that $\bP(\mathcal E_L^c) \le C\exp(-Kt^{1-2/d})$, which is always possible by \eqref{prob.E}.

Next, for $\delta>0$, and $I$ some integer, define 
$$\cR_n(I,\delta) := \bigcup_{i\le I} \cR_n(r_i,\delta \rho_i).$$
We claim that one can find $\delta\in (0,1)$ and $I\ge 0$, such that 
\begin{equation}\label{goal.2}
\mathcal E_L \subseteq \{\sup_{x\in \Z^d} G(\cR_n\bs \cR_n(I,\delta) - x) \le \epsilon t^{2/d}\}. 
\end{equation} 
This would conclude the proof, since for any fixed $I$ and $\delta$, one can find $A$ and $\rho$, such that 
$$\cR_n(I,\delta) \subseteq \cR_n(At^{1/d}, \rho). $$ 
So let us prove \eqref{goal.2} now. Fix some $x\in \Z^d$, and consider the decomposition of space into concentric shells $(\mathcal S_k)_{k\ge 0}$, as in the previous subsection. By Lemma \ref{lem.sumGreen}, one has 
$$G((\cR_n\bs \cR_n(I,\delta) ) \cap \S_0) \le C\delta^{2/d} r_0^2\le C\delta^{2/d} t^{2/d},$$
and for any $1\le k\le I$, by \eqref{Green}, 
$$G((\cR_n\bs \cR_n(I,\delta) ) \cap \S_k)\le \frac{C\delta \rho_k r_k^d}{r_{k-1}^{d-2}} \le C\delta \frac{\log n}{r_k^{d-4}}.$$
Thus, summing over $k\le I$, yields 
$$\sum_{1\le k\le I}  G((\cR_n\bs \cR_n(I,\delta) ) \cap \S_k)\le C\delta \frac{\log n}{r_0^{d-4}} \le C\delta t^{2/d}. $$ 
On the other hand, since for any $\delta \le 1$, $\cup_{i\le I} \Lambda_i \subseteq \cR_n(I,\delta)$, 
one has for any $k> I$, 
$$G((\cR_n\bs \cR_n(I,\delta) ) \cap \S_k) \le \sum_{j=I+1}^k G(\Lambda_j\cap \S_k) + G(\Lambda_{k+1}^*\cap \S_k),$$
and the same bounds as in the previous subsection give on $\mathcal E_L$,
$$ \sum_{k\ge I+1} G((\cR_n\bs \cR_n(I,\delta) ) \cap \S_k) \le C t^{2/d} \sum_{j\ge I+1} \rho_j^{\frac{d-4}{d-2}} +C\sum_{k\ge I+1}\frac{\log n}{r_k^{d-4}} \le C\rho_I^{\frac {d-4}{d-2}} t^{2/d}.$$
Altogether, we see that by choosing $I$ large enough, and $\delta$ small enough, we get \eqref{goal.2}, concluding the proof of \eqref{borne.simple.2}. 
Finally the fact that $1/\rho$ and $A$ can be chosen, so that they grow at most polynomially in $1/\epsilon$ is by construction.


\section{Proof of Propositions \ref{prop.localtime} and \ref{prop.finiteballs}} \label{sec.propfiniteballs}
\subsection{Proof of \eqref{bornes.localtime} and \eqref{bornes.intersections}.} 
We start with the lower bounds. Note that it suffices to do it for the intersection of two ranges, 
that is for \eqref{bornes.intersections}, and for a finite time horizon. For this we use Proposition 4.1 from \cite{AS17}, which entails the following fact:   

\begin{proposition}[\cite{AS17}] \label{prop-as17} Assume $d\ge 3$. There are positive constants  
$\rho$, $\kappa$ and $C$, such that for $n$ large enough, for 
 any subset $\Lambda \subseteq Q(0,n^{1/d})$, with $|\Lambda|>C$, one has 
\begin{equation*}
\bP(|\cR_n\cap \Lambda |> \rho |\Lambda|) \ge 
\exp(-\kappa \cdot n^{1-2/d}).
\end{equation*}
\end{proposition}
Note that Proposition 4.1 in \cite{AS17} is stated for dimension $3$ only, but its proof applies mutatis mutandis in higher dimension.

Now for $\alpha=1/\rho^2$ we force, at a cost given by Proposition~\ref{prop-as17},
the range $\tilde \cR_{\alpha t}$
to cover a fraction $\rho$ of $Q(0,r)$ with $r=(\alpha t)^{1/d}$,
and in turn force $ \cR_{\alpha t}$ to cover a fraction $\rho$  
of $\tilde \cR_{\alpha t}\cap Q(0,r)$. Observe that one has the inclusion 
\[
\{|\tilde \cR_{\alpha t}\cap Q(0,r)|> \rho r^d\}\cap
\{|\cR_{\alpha t}\cap\tilde \cR_{\alpha t}\cap Q(0,r)|> \rho  
|\tilde \cR_{\alpha t}\cap Q(0,r)|\}\subseteq 
\{|\cR_{\alpha t}\cap\tilde \cR_{\alpha t}|> \rho^2 r^d=t\}, 
\]
which concludes the proof of the lower bounds.

Concerning the upper bounds, as was already mentioned, they simply follow from \eqref{finite.time}, \eqref{borne.simple} (say with $\beta=1$), 
together with Lemma \ref{lem.tail}.

\subsection{Proof of \eqref{finite.balls.localtime} and \eqref{scenario.finiteballs}}
We first state and prove a corollary of Theorem \ref{theo.finitecovering} (and the remark following it), which might be of general interest. 
Recall that $\mathcal X_r$ is the collection of finite subsets of $\Z^d$, whose points are at distance at least $r$ from each other, and for $N$ positive integer, let $\mathcal X_{r,N}$ be the subset of $\mathcal X_r$ formed by subsets of cardinality $N$.  

\begin{proposition}\label{cor.theo.covering}
Let $\{S_n\}_{n\ge 0}$ and $\{\tilde S_n\}_{n\ge 0}$ be two independent simple random walks
on $\Z^d$, $d\ge 5$. There exist positive constants $\kappa$ and $C$, such that for any integers $r$ and $N$, and any $\rho>0$, satisfying 
\begin{equation}\label{weak-hyp}
\rho r^{d-2}> CN^{2/d}\log N,
\end{equation}
one has 
\begin{equation}\label{borne.cor.covering}
\bP\left(\exists \mathcal C\in \mathcal X_{4r,N}  :  \ell_\infty(Q(x,r)) > \rho r^d,\,  
 \tilde \cR_\infty \cap Q(x,r) \neq \emptyset, \, \forall x\in \mathcal C\right)  \le C\exp(-\kappa\,  \rho r^{d-2}  N^{1-\frac 2d}).
\end{equation}
\end{proposition}
An important difference here with the statement of Theorem \ref{theo.finitecovering} is that the set $\C$ is not fixed in advance anymore, but this is compensated  by 
the fact that we impose to another independent walk to visit all the cubes centered at points of $\C$. 
\begin{proof}[Proof of Proposition \ref{cor.theo.covering}]
Note that by replacing $r$ by $2r$, $\rho$ by $\rho/2^d$, and $N$ by $\lceil N/2\rceil$ if necessary, one can consider only subsets $\mathcal \C$ whose points belong to $2r\Z^d\bs \{0\}$. 
Fix now such set $\C\in \mathcal X_{4r,N}$, and denote by $x_1,\dots,x_N$ its elements. Note first that for any $r$ and $\rho$ satisfying \eqref{weak-hyp}, with $C$ large enough, 
Theorem \ref{theo.finitecovering} (and the remark following it) yield for some constant $\kappa$, 
\begin{equation}\label{many-2}
\bP(\ell_\infty(Q(x,r)) > \rho r^d, \ \forall x\in \C) \le C\exp(-\kappa \rho r^{d-2} 
N^{1-\frac 2d}).
\end{equation}
On the other hand, by \eqref{bound.hit} one has  
\begin{equation}\label{many-4}
\bP\big(\cR_\infty \cap Q(x,r) \neq \emptyset,\ \forall x\in \C\big)\le  ( C r^{d-2})^N\cdot  G(x_1,\dots,x_N), 
\end{equation}
where, denoting by $ \mathfrak{S}_N$ the set of permutations 
of $\{1,\dots,N\}$, 
\begin{equation*}
G(x_1,\dots,x_N):=
\sum_{\sigma\in \mathfrak{S}_N}
G(x_{\sigma_1})\prod_{i=1}^{N-1}G\big(x_{\sigma_{i+1}}-x_{\sigma_i}\big).
\end{equation*}
For any $2>q>1$, using H\"older's inequality
\begin{equation}\label{many-6}
\begin{split}
\sum_{x_1,\dots,x_N \in 2r\Z^d\bs \{0\}} G^q(x_1,\dots,x_N)& \le 
\sum_{x_1,\dots,x_N \in 2r\Z^d\bs \{0\}} (N!)^{q-1} \sum_{\sigma\in \mathfrak{S}_N}
G^q(x_{\sigma_1})\prod_{i=1}^{N-1} G^q(x_{\sigma_{i+1}}- x_{\sigma_i})\\
& \le  (N!)^{q} \Big(\sum_{z\in 2r\Z^d\bs \{0\}} G^q(z)\Big)^N.
\end{split}
\end{equation}
Now fix some $2>q>\frac d{d-2}$,  
and note that by \eqref{Green}, one has (with a possibly larger constant $C$), 
$$\sum_{z\in 2r\Z^d\bs \{0\}} G^q(z)\le Cr^{q(2-d)},$$
so that \eqref{many-4} and \eqref{many-6} give, 
\begin{equation}\label{many-7}
\sum_{x_1,\dots,x_N \in 2r\Z^d\bs \{0\}}\bP^q\big(\cR_\infty \cap Q(x,r) \neq \emptyset,\ \forall x\in \C\big)\le C^{2N}\cdot (N!)^q.
\end{equation}
Then \reff{many-2} and \eqref{many-7} yield  
\begin{equation*}\label{many-8}
\begin{split}
\bP\big(\exists \mathcal C\in \mathcal X_{4r,N}  & :  \ell_\infty(Q(x,r)) > \rho r^d,\,  
 \tilde \cR_\infty \cap Q(x,r) \neq \emptyset, \, \forall x\in \mathcal C\big)\\
& \le \sum_{\C\in \mathcal X_{4r,N}}\bP\big(\ell_\infty(Q(x,r)) > \rho r^d, \ \forall x\in \C\big)\times 
\bP\big(\cR_\infty \cap Q(x,r) \neq \emptyset,\ \forall x\in \C\big)
\\
&\le \sum_{\C\in \mathcal X_{4r,N}}\bP^{2-q}\big(\ell_\infty(Q(x,r)) > \rho r^d, \ \forall x\in \C\big)\times 
\bP^q\big(\cR_\infty \cap Q(x,r) \neq \emptyset,\ \forall x\in \C\big)
\\
&\le C^{2N} (N!)^{q} \cdot 
\exp(-\kappa(2-q) \rho r^{d-2} N^{1-\frac 2d}), 
\end{split}
\end{equation*}
and we conclude the proof using the hypothesis \eqref{weak-hyp}. 
\end{proof}

One can now conclude the proofs of \eqref{finite.balls.localtime} and \eqref{scenario.finiteballs}. First we choose $\beta$ large enough, so that the probability of the 
event $\{ \tilde \ell_\infty(\cR[n,\infty))\ge 1\}$ is negligible, when we take $n=\exp(\beta t^{1-2/d})$, which is always possible by \eqref{finite.time} and the lower bound in \eqref{bornes.intersections}.

Next, by Lemma \ref{lem.tail} and \eqref{borne.simple.2}, it suffices to show that for any fixed $A>0$ and $\rho \in (0,1)$, 
the set $\cR_n(At^{1/d},\rho)\cap \tilde \cR_\infty$ can be covered by at most $N$ disjoint cubes of side length $At^{1/d}$, 
for some well-chosen constant $N\in \N$. To see this, we first fix the constant $N$ large enough, such that the bound obtained in \eqref{borne.cor.covering} with $r=At^{1/d}$, is negligible when 
compared to the lower bound in \eqref{bornes.intersections}.

Then we define inductively a sequence of boxes as follows. First if the set $\cR_n(At^{1/d},\rho)\cap \tilde \cR_\infty$ is nonempty, pick some point $x_1$ in it. 
Then, if the set $\cR_n(At^{1/d},\rho)\cap \tilde \cR_\infty \cap Q(x_1,4At^{1/d})^c$ is empty, stop the procedure. Otherwise pick some $x_2$ in it, and continue like this until we 
exhaust all points of $\cR_n(At^{1/d},\rho)\cap \tilde \cR_\infty$. Note that the points we define by this procedure $x_1,x_2,\dots$ are all at distance at least $4At^{1/d}$ one from each other by definition. 
Furthermore, for each $i$, one has by definition $|Q(x_i,At^{1/d})\cap \cR_n| \ge \rho A^dt$. Thus by Proposition \ref{cor.theo.covering}, the probability 
that we end up with more than $N$ cubes is negligible. Finally this means that with (conditional) probability going to $1$, as $t\to \infty$, 
we can cover $\cR_n(At^{1/d},\rho)\cap \tilde \cR_\infty$ by at most $N$ cubes of side length $4At^{1/d}$, which concludes the proofs of \eqref{finite.balls.localtime} and \eqref{scenario.finiteballs} (since each such 
cube is in turn the union of only a fixed number of cubes of side length $t^{1/d}$). Moreover, if $A$ and $1/\rho$ grow at most polynomially in $1/\epsilon$, then $N$ also by construction. 
 

\section{Proof of Theorems \ref{theo.main} and \ref{theo.scenario}} \label{sec.proofstheo} 
Let us define $\mathcal I_\infty:= \lim_{b\to \infty} \mathcal I(b)$, with $\mathcal I(b)$ as in \eqref{BBH.Phet}. 
Since it is easier to realize a large intersection in infinite time, rather than in any finite time, we already know that 
\begin{equation}\label{lower.bound.optimale}
\liminf_{t\to \infty} \frac 1{t^{1-\frac 2d}} \log \bP(|\cR_\infty\cap \tilde \cR_\infty|\ge t) \ge -\mathcal I_\infty.
\end{equation}
The proofs of Theorems \ref{theo.main} and \ref{theo.scenario} are now based on the following result.  
\begin{proposition}\label{prop.concav}
For $k$, $L$, and $t$ some positive integers, and $\delta\in (0,1)$ some real, define 
\begin{equation*}
\A(k,L,\delta, t):=\left\{\exists x_1,\dots,x_k\in \Z^d\,:\, 
\begin{array}{l}
\|x_i-x_j\|\ge L^2t^{1/d} \quad \forall i\neq j \\
|\cR_\infty \cap \tilde \cR_\infty \cap Q(x_i,Lt^{1/d})|\ge \delta t \quad \forall 1\le i\le k \\
|\cR_\infty \cap \tilde \cR_\infty \cap (\bigcup_{i=1}^k Q(x_i,Lt^{1/d}))| \ge t 
\end{array}
\right\}. 
\end{equation*}
There exist $C>0$ and $L_0\ge 1$, such that for any $L\ge L_0$, $k\le L$, and $\delta\in (0,1)$, 
\begin{equation*}
\limsup_{t\to \infty} \frac 1{t^{1-\frac 2d}}\log \bP(\A(k,L,\delta,t) ) \le -\mathcal I_\infty\left(1+(1-\frac 1{2^{2/d}})[(k-1)\delta]^{1-2/d}\right) + \frac{C\log k}{\log L}. 
\end{equation*}
\end{proposition}
Note that Theorem \ref{theo.main} follows from 
Theorem \ref{theo.scenario} and Proposition \ref{prop.concav}, applied with $k=1$. Now before we prove Proposition \ref{prop.concav}, let us see how it allows to prove Theorem \ref{theo.scenario} as well.

\begin{proof}[Proof of Theorem \ref{theo.scenario}]
For $N\ge 1$ some integer and $t>0$, define the event 
\begin{equation*}
\mathcal B_{N,t} := \left\{\exists x_1,\dots,x_N\in \Z^d\, :\, |\cR_\infty \cap \tilde \cR_\infty \cap \left(\bigcup_{i=1}^N Q(x_i,t^{1/d})\right)| \ge t\right\},
\end{equation*}
and for $L \ge 1$ another integer, set 
\begin{equation}
\mathcal B_{N,L,t} := \left\{\exists x_1,\dots,x_N\in \Z^d\, :\, 
\begin{array}{l}
\|x_i-x_j\|\ge L^2t^{1/d}\quad \forall i\neq j \\
|\cR_\infty \cap \tilde \cR_\infty \cap \left(\bigcup_{i=1}^N Q(x_i,Lt^{1/d})\right)| \ge t
\end{array}
\right\}.
\end{equation}
We claim that for any $N\ge 1$, $L_0\ge 1$, and $t>0$, one has 
\begin{equation}\label{inclusion.B}
\mathcal B_{N,t} \subseteq  \{B_{N,L,t}\, :\, L=L_0,\dots,(2L_0)^{2^N}\}.  
\end{equation}  
Indeed, assume $\B_{N,t}$ holds, and consider $x_1,\dots,x_N$ realizing this event. Let also $I_0:=\{1,\dots,N\}$. If the $(x_i)_{i\in I_0}$ are all at distance at least $L_0^2t^{1/d}$ one from each other, we stop and $\B_{N,L_0,t}$ 
holds. If not, consider the first index $i$, such that $x_i$ is at distance smaller than $L_0^2t^{1/d}$ from one of the $x_j$, with $j<i$, and set $I_1=I_0\bs\{i\}$. Set also $L_1 =(2L_0)^2$, and restart the algorithm with $I_1$ and $L_1$ in place of $I_0$ and $L_0$ respectively. Since this procedure stops in at most $N$ steps, we deduce well \eqref{inclusion.B} (note that we may end up with less than $N$ points,
but since we do not impose the intersection of the ranges with all cubes being nonempty, we may always add arbitrary some distant points at the end).
Next, let $K>0$ be some fixed constant. We claim that for any reals $\epsilon \in (0,1)$, $t>0$, and any integers $N\le \epsilon^{-K}$, $L\ge 1$, one has 
\begin{equation}\label{inclusion.B.A}
\mathcal B_{N,L,t} \subseteq \bigcup_{k=1}^N \mathcal A\left(k,L,\frac{\epsilon^{\frac d{d-1}}}{2^{(d-1)K}k},(1-\epsilon)t\right). 
\end{equation}
To see this, assume that the event $\B_{N,L,t}$ holds, and consider $x_1,\dots,x_N$ realizing it. Set $k_0=N$, and $J_0= \{1,\dots,N\}$, and then let 
$$J_1:=\{i \in J_0 \, :\, |\cR_\infty \cap \tilde \cR_\infty \cap Q(x_i,Lt^{1/d})|\ge \frac{\epsilon}{2k_0}\}. $$
Note that by definition of $\B_{N,L,t}$ and $J_1$, 
$$|\cR_\infty \cap \tilde \cR_\infty \cap (\bigcup_{i\in J_1} Q(x_i,Lt^{1/d}))|\ge (1-\frac{\epsilon}{2})t.$$
Thus if $|J_1|\ge \epsilon^{\frac 1{d-1}}k_0$, we are done, since in this case 
$\mathcal B_{N,L,t}  \subset \mathcal A(k_1,L,\frac{\epsilon^{\frac d{d-1}}}{2k_1},(1-\frac{\epsilon}{2})t)$, with $k_1:=|J_1|$.   If not, define 
$$J_2:= \{i\in J_1\, :\, |\cR_\infty \cap \tilde \cR_\infty \cap Q(x_i,Lt^{1/d})|\ge \frac{\epsilon}{4k_1}\}. $$
One has by definition, 
$$|\cR_\infty \cap \tilde \cR_\infty \cap (\bigcup_{i\in J_2} Q(x_i,Lt^{1/d}))|\ge (1-\frac{\epsilon}{2}-\frac{\epsilon}4)t.$$
Thus if $|J_2|\ge \epsilon^{\frac 1{d-1}}k_1$, we are done as well, and if not we continue defining inductively $(J_i)_{i\ge 1}$ and $(k_i)_{i\ge 1}$ as above, until either $|J_i|\ge \epsilon^{\frac 1{d-1}}k_{i-1}$,  or $|J_i|=1$, for some $i$. Note that in the latter  
case one has $\mathcal B_{N,L,t} \subseteq  \mathcal A(1,L,1-\epsilon,(1-\epsilon)t)$.  
Since on the other hand at each step we reduce 
the cardinality of the set of points by a factor at least $\epsilon^{1/(d-1)}$, and by hypothesis $N\le \epsilon^{-K}$, this algorithm must stop in at most $(d-1) K$ steps, 
and this proves well \eqref{inclusion.B.A}.

Recall next that Proposition \ref{prop.finiteballs} says that for any $\epsilon$, there exists some integer $N=N(\epsilon)$, such that 
$$\lim_{t\to \infty} \bP(\B_{N,(1-\epsilon)t} \mid |\cR_\infty \cap \tilde \cR_\infty|\ge t) =1, $$
and furthermore, that one can find a constant $K$, such that $N(\epsilon)\le \epsilon^{-K}$, at least for $\epsilon$ small enough.  
Moreover, the constant $K$ being fixed, Proposition \ref{prop.concav} and the lower bound \eqref{lower.bound.optimale} also show that for any 
$\epsilon$ small enough, any $L\ge \exp(1/\epsilon)$, and $2\le k\le L$, 
$$\lim_{t\to \infty} \bP(\mathcal A(k,L,\frac{\epsilon^{\frac{d}{d-1}}}{2^{(d-1)K}k},(1-\epsilon)^2t) \mid |\cR_\infty \cap \tilde \cR_\infty|\ge t) =0.$$
Thus Theorem \ref{theo.scenario} follows from \eqref{inclusion.B} and \eqref{inclusion.B.A}, taking $L_0\ge \exp(1/\epsilon)$, and noting that for any $L\le L'$, and $\delta \le 1$, one has the inclusion $\mathcal A(1,L,\delta,t)\subseteq \B_{1,L',t}$.  
\end{proof}

It remains now to prove Proposition \ref{prop.concav}. For this we will need the following lemma. 

\begin{lemma}\label{lem.concav}
Assume $q\in (0,1]$. For any integer $k\ge 1$, 
and $t_1,\dots,t_k$ positive numbers, we have
\begin{equation}\label{ineq-conca}
t_1^q+\dots+t_k^q\ge \big( \sum_{i=1}^k t_i\big)^q+
(1-\frac{1}{2^{1-q}})\Big((k-1)\min_{i\le k} (t_i)\Big)^q. 
\end{equation}
\end{lemma}
\begin{proof}
The proof is by induction. For $k=2$, assume $t_1\ge t_2>0$. Then 
\reff{ineq-conca} reduces to seeing that
\[
t_1^q+\frac{1}{2^{1-q}} t_2^q\ge (t_1+t_2)^q.
\]
If we set $x=t_2/t_1$, we need to show that for $0\le x\le 1$, 
\[
1+\frac{x^q}{2^{1-q}}\ge (1+x)^q.
\]
By taking derivatives of the two terms, the problem reduces to
checking that $2x<1+x$ for $0\le x\le 1$, which is indeed true.
The induction follows: set $\alpha=1-1/2^{1-q}$, and write
\[
\begin{split}
t_k^q+\sum_{i=1}^{k-1} t_i^q& \ge  t_k^q+ \big( \sum_{i=1}^{k-1} t_i\big)^q
+\alpha\Big((k-2)\min_{i\le k-1} (t_i)\Big)^q\\
& \ge \big(\sum_{i=1}^{k} t_i\big)^q
+\alpha\Big\{\min(t_k,\sum_{i=1}^{k-1} t_i)^q+\Big((k-2)\min_{i\le k-1} (t_i)
\Big)^q\Big\}\\
& \ge  \big(\sum_{i=1}^{k} t_i\big)^q+ \alpha 
\Big\{\min_{i\le k}(t_i)^q+\Big((k-2)\min_{i\le k-1} (t_i)
\Big)^q\Big\}\\
& \ge  \big(\sum_{i=1}^k t_i\big)^q+
\alpha\Big((k-1)\min_{i\le k} (t_i) \Big)^q,
\end{split}
\]
using the inequality $a^q + b^q\ge (a+b)^q$ at the last line. 
\end{proof}

\begin{proof}[Proof of Proposition \ref{prop.concav}]
The idea is to cut the two trajectories $(S_n)_{n\ge 0}$ and $(\tilde S_n)_{n\ge 0}$ realizing the event $\mathcal A(k,L,\delta,t)$ into 
excursions in a natural way, and then realizing some surgery, to 
compare the probability of the event to the product of the probabilities of realizing a certain intersection inside $k$ different cubes.  
Now let us proceed with the details. Fix $x_1,\dots,x_k\in \Z^d$, with $\|x_i-x_j\|\ge L^2t^{1/d}$, for all $i\neq j$. 
For $1\le i \le k$, set $Q_i:= Q(x_i,Lt^{1/d})$, and $\overline Q_i:= Q(x_i,L^2t^{1/d})$. 
Assume to simplify notation that all the $x_i$ belong to $\lfloor L^2t^{1/d}\rfloor\Z^d$ (if not one can always replace them by 
the closest points on this lattice, and increase the side-length of the cubes $Q_i$, and reduce the one of the $\overline Q_i$, 
both by an innocuous factor $2$). Finally to simplify also the discussion below, we further assume that the origin does not belong to any of 
the cubes $\overline Q_i$ (minor modifications of the argument would be required otherwise, which we safely leave to the reader). 
Then define two sequences of stopping times $(s_\ell)_{\ell \ge 0}$ and $(\tau_\ell)_{\ell \ge 0}$ as follows. First $s_0=\tau_0=0$, and for $\ell\ge 1$,  
$$\tau_\ell :=\inf\{n\ge s_{\ell-1}  : S_n \in \bigcup_{i=1}^k \partial Q_i\},\quad \text{and}\quad s_\ell := \inf \{n\ge \tau_\ell  : S_n \in \bigcup_{i=1}^k \partial {\overline Q}_i\}. $$
Let $\mathcal N:= \sum_{\ell =1}^\infty \1\{\tau_\ell <\infty\}$, be  
the total number of excursions. Let $\tau(\Lambda):=\inf\{n : S_n\in \Lambda\}$, for the hitting time of a subset $\Lambda\subseteq \Z^d$. It follows from \eqref{Green} and \eqref{bound.hit}, that for any $\ell\ge 1$, 
\begin{align*}
\bP(\tau_{\ell+1}<\infty\mid \tau_\ell<\infty)  \le \sup_{1\le i\le k} 
\sup_{y\in \partial {\overline Q}_i } \bP_y(\tau(\cup_{i=1}^k Q_i)<\infty) \le  \sup_{1\le i\le k} 
\sup_{y\in \partial {\overline Q}_i }  \sum_{j=1}^k \bP_y(\tau(Q_j)<\infty) \le \frac{C}{L^{d-3}},
\end{align*}
for some constant $C>0$, using also the hypothesis $k\le L$, for the last inequality. Consequently, one has for some constant $C_0>0$, and all $t$ large enough, 
\begin{equation}\label{nbr.exc}
\bP\left(\mathcal N\ge \frac{C_0 t^{1-\frac 2d}}{\log L} \right) \le \exp(-2\mathcal I_\infty\cdot t^{1-\frac 2d}). 
\end{equation}
Now, let $i(\ell)$ be the index of the cube to which $S(\tau_\ell)$ belongs, when $\tau_\ell$ is finite: that is $S(\tau_\ell) \in Q_{i(\ell)}$.  
Define further $\ell_1,\dots,\ell_k$ inductively by $\ell_1=1$, and for $j\ge 1$, 
$$\ell_{j+1} = \inf \left\{\ell >\ell_j : i(\ell) \notin \{i(\ell_1),\dots,i(\ell_j)\}\right\}. $$
This induces a permutation $\sigma \in \mathfrak S_k$, defined by $\sigma(j) := i(\ell_j)$, which represents the order of first visits of the cubes by the walk. 
Recall now the definition of the harmonic measure $\mu_i$ of $Q_i$: 
$$\mu_i(z) := \bP_z[\cR[1,\infty) \cap Q_i=\emptyset ],\quad \forall z\in \partial Q_i.$$ 
We will need the following estimate (see Proposition 6.5.4 in \cite{LL}): for $y\notin Q_i$, and $z\in \partial Q_i$, 
\begin{equation} \label{hitting.harmonic}
\bP_y[S_{\tau(Q_i)} = z \mid \tau(Q_i)<\infty] = \mu_i(z)\left[1+\mathcal O\left(\frac{Lt^{1/d}}{\|y-x_i\|}\right)\right].
\end{equation}
Combining it with \eqref{Green} and \eqref{bound.hit}, this yields for some constant $c_1>0$,  for any $1\le i,j\le k$, and any $z\in \partial Q_j$,  
\begin{equation}\label{tau.bound.1}
\sup_{y\in \partial {\overline Q}_i} \bP_y(\tau(Q_j)<\infty, S_{\tau(Q_j)}=z) \le \frac{c_1}{L^{d-2}}\mu_j(z), 
\end{equation}
and when $i\neq j$, we also get 
\begin{equation}\label{tau.bound.2}
\sup_{y\in \partial {\overline Q}_i} \bP_y(\tau(Q_j)<\infty, S_{\tau(Q_j)}=z) \le \frac{c_1}{L^{d-2}}\mu_j(z)\cdot (L^2t^{1/d})^{d-2}G(x_j-x_i). 
\end{equation}
Define analogously $\tilde \tau_\ell$, $\tilde s_\ell,\tilde i(\ell),\dots$, for the walk $\tilde S$. 
Then for $1\le j\le k$, set 
$$\mathcal I_j := \Big|\Big(\bigcup_{\ell \, :\, i(\ell) = j} \cR[\tau_\ell, s_\ell] \Big) \cap \Big(\bigcup_{\ell\,  :\, \tilde i(\ell) = j} \tilde \cR[\tilde \tau_\ell,\tilde s_\ell] \Big)\Big|, $$
the number of intersections of the two walks inside the $Q_j$. Note that by construction, 
$$\mathcal I_j =|\cR_\infty \cap \tilde \cR_\infty \cap Q_j|, \quad \text{for all } 1\le j \le k.$$
Let now $t_1,\dots,t_k$, and $n$, $m$ be some fixed positive integers. Then consider two fixed 
sequences of indices $(i_1,\dots,i_n)$ and $(\tilde i_1,\dots,\tilde i_m)$, taking values in $\{1,\dots,k\}$, such that all $j\in \{1,\dots,k\}$ appear at least once in the two sequences. This induces two permutations $\sigma, \tilde \sigma\in \mathfrak S_k$, as defined above (one for each sequence). Then set 
$$G_\sigma(x_1,\dots,x_k):=(L^2t^{1/d})^{k(d-2)} \cdot G(x_{\sigma(1)})\prod_{j=1}^{k-1}G(x_{\sigma(j+1)} - x_{\sigma(j)}). $$
Let also for $1\le j\le k$, 
$$n_j := \sum_{\ell = 1}^n \1\{i_\ell = j\}, \quad \text{and} \quad m_j :=  \sum_{\ell = 1}^m \1\{\tilde i_\ell = j\}. $$
Then applying \eqref{tau.bound.1}, and \eqref{tau.bound.2} at indices $\ell_j$, for $1\le j\le k$, shows that 
\begin{align}\label{main.bound.1}
\nonumber & \bP\left(
\begin{array}{c} 
\mathcal N = n, \quad  \tilde {\mathcal N} = m, \quad \mathcal I_j \ge t_j, \ \forall j=1,\dots,k\\
 i(\ell) = i_\ell\  \forall \ell\le n, \quad \text{and}\quad \tilde i(\ell) = \tilde i_\ell \ \forall \ell \le m 
\end{array}
 \right) \\
 & \le (\frac {c_1}{L^{d-2}})^{n+m} \left(\prod_{j=1}^k \bP_{\mu_j,n_j,m_j}(\mathcal I_j\ge t_j) \right) G_\sigma(x_1,\dots,x_k)G_{\tilde \sigma}(x_1,\dots,x_k). 
 \end{align}
where for all $1\le j\le k$, $\bP_{\mu_j,n_j,m_j}$ denotes the law of the walk conditionally on 
$(S(\tau_\ell))_{\ell :i_\ell =j}$, and $(\tilde S(\tilde \tau_\ell))_{\ell :\tilde i_\ell = j}$, 
being independent and identically distributed with joint law $\mu_j$, or equivalently the law of $n_j+m_j$ independent 
excursions starting from law $\mu_j$.

Our next task is to bound the probabilities $\bP_{\mu_j,n_j,m_j}(\mathcal I_j\ge t_j)$, using \eqref{BBH.Phet}. 
Proposition 6.5.1 in \cite{LL} shows that for some constant $c>0$, for any $1\le j\le k$,  and $y\notin Q_j$, 
$$\bP_y(\tau(Q_j) <\infty) =c\frac{\capa(Q_j)}{\|y-x_j\|^{d-2}} \left[1+\mathcal O\left(\frac{Lt^{1/d}}{\|y-x_j\|}\right)\right], $$
where $\capa(Q_j)$ denotes the capacity of the box $Q_j$, for which all we need to know is that it is of order $L^{d-2}t^{1-2/d}$.  
When combined with \eqref{hitting.harmonic} this yields the existence of a constant $c_2>0$, such that for all $1\le j\le k$, and all $z\in \partial Q_j$, 
\begin{equation}\label{borne.inf.hitQj}
\inf_{y\in \partial {\overline Q}_j} \bP_y(\tau(Q_j)<\tau(Q(x_j,L^3t^{1/d})),\, S_{\tau(Q_j)}=z) \ge \frac{c_2}{L^{d-2}}\mu_j(z).
\end{equation}
Now let $x\in \Z^d$, be such that the origin belongs to $\partial Q(x,L^2t^{1/d})$. The above inequality \eqref{borne.inf.hitQj} 
shows that for any $1\le j\le k$, and any integers $n_j, m_j$, 
\begin{equation}\label{bound.box}
\bP(|\cR_{\tau(Q(x,L^3t^{1/d})} \cap \tilde \cR_{\tilde \tau(Q(x,L^3t^{1/d})} \cap Q(x,Lt^{1/d})|\ge t_j)\ge \left(\frac{c_2}{L^{d-2}}\right)^{n_j+m_j} \bP_{\mu_j,n_j,m_j}(\mathcal I_j\ge t_j).
\end{equation}
On the other hand, Lemmas \ref{lem.tail} and \ref{lem.sumGreen} show that for some constant $b>0$,  
$$\bP(\tau(Q(x,L^3t^{1/d}))>bt) \le \exp(-2\mathcal I_\infty t^{1-\frac 2d}),$$
at least for $t$ large enough. Thus if $t_j \ge \delta t$, we get with \eqref{BBH.Phet}, that at least for $t$ large enough, the 
left-hand side of \eqref{bound.box} is bounded above by $2\bP(|\cR_{bt}\cap \tilde \cR_{bt}|\ge t_j)$. 
When combined with \eqref{main.bound.1}, this shows that for some constant $b>0$, for all $t_j\ge \delta t$, 
\begin{align}\label{main.bound.2}
\nonumber & \bP\left(
\begin{array}{c} 
\mathcal N = n, \quad  \tilde {\mathcal N} = m, \quad \mathcal I_j \ge t_j, \ \forall j=1,\dots,k\\
 i(\ell) = i_\ell\  \forall \ell\le n, \quad \text{and}\quad \tilde i(\ell) = \tilde i_\ell \ \forall \ell \le m 
\end{array}
 \right) \\
 & \le 2^k(\frac{c_1}{c_2})^{n+m} \left(\prod_{j=1}^k \bP(|\cR_{bt} \cap \tilde \cR_{bt}|\ge t_j)  \right) G_\sigma(x_1,\dots,x_k)G_{\tilde \sigma}(x_1,\dots,x_k)\\
 \nonumber & \le 2^k (\frac{c_1}{c_2})^{n+m} \left(\prod_{j=1}^k \bP(|\cR_{b't_j} \cap \tilde \cR_{b't_j}|\ge t_j)  \right) \max_{\sigma\in \mathfrak S_k} G_\sigma(x_1,\dots,x_k)^2,  
 \end{align}
 with $b'=b/\delta$. 
Summing over all possible sequences $(i_\ell)_{\ell \le n}$ and $(\tilde i_\ell)_{\ell \le m}$, we get 
\begin{align*}
&\bP\left(\mathcal N = n, \, \tilde {\mathcal N} = m, \, \mathcal I_j \ge t_j,  \, \forall j=1,\dots,k  \right) \\
&\qquad  \le 2^k  (\frac{kc_1}{c_2})^{n+m} \left(\prod_{j=1}^k \bP(|\cR_{b't_j} \cap \tilde \cR_{b't_j}|\ge t_j)  \right) \max_{\sigma\in \mathfrak S_k} G_\sigma(x_1,\dots,x_k)^2.  
 \end{align*}
Summing then over all $n,m\le N_0:=\lfloor \frac{C_0t^{1-\frac 2d}}{\log L}\rfloor$, with $C_0$ as in \eqref{nbr.exc}, we get 
\begin{align}\label{main.bound.3}
\nonumber & \bP\left(\mathcal N \le N_0, \, \tilde {\mathcal N} \le N_0, \, \mathcal I_j \ge t_j,  \, \forall j=1,\dots,k  \right) \\
& \qquad \le 2^k N_0^2 (\frac{kc_1}{c_2})^{2N_0} \left(\prod_{j=1}^k \bP(|\cR_{b't_j} \cap \tilde \cR_{b't_j}|\ge t_j)  \right) \max_{\sigma\in \mathfrak S_k} G_\sigma(x_1,\dots,x_k)^2.   
 \end{align}
Now letting $r:= \lfloor L^2t^{1/d}\rfloor$, we get using \eqref{Green}, 
\begin{equation*}
 \sum_{x_1,\dots,x_k\in r\Z^d} \max_{\sigma \in \mathfrak S_k} G_\sigma(x_1,\dots,x_k)^2 
 \le \sum_{\sigma \in \mathfrak S_k} \sum_{x_1,\dots,x_k\in r\Z^d} G_\sigma(x_1,\dots,x_k)^2 \le C^k k!. 
\end{equation*}
Thus summing over all $x_1,\dots,x_k\in r\Z^d$ in \eqref{main.bound.3}, and using \eqref{nbr.exc}, we get 
\begin{align*}
&\sum_{x_1,\dots,x_k \in r\Z^d} \bP( \mathcal I_j \ge t_j,   \forall j=1,\dots,k  ) \\
 & \qquad \le (2C)^k(k!) N_0^2 (\frac{kc_1}{c_2})^{2N_0}\left(\prod_{j=1}^k \bP(|\cR_{b't_j} \cap \tilde \cR_{b't_j}|\ge t_j)  \right)  + \exp(-2\mathcal I_\infty t^{1-\frac 2d}). 
 \end{align*}
Finally by using \eqref{BBH.Phet} and Lemma \ref{lem.concav} (with $q=1-\frac 2d$), and then summing over all possible $t_1,\dots,t_k\ge \delta t$, satisfying $t_1+\dots + t_k = t$, we conclude the proof of the proposition. 
 \end{proof}

\section*{Acknowledgements} This research 
was supported by public grants overseen by the 
french National Research Agency, ANR SWiWS (ANR-17-CE40-0032-02) and ANR MALIN (ANR-16-CE93-0003).

\end{document}